\documentclass[a4paper,reqno,12pt]{amsart} 

\usepackage[english]{babel}
\usepackage[utf8]{inputenc}
\usepackage[T1]{fontenc}
\usepackage{eucal}
\usepackage{xparse}
\usepackage{microtype}
\usepackage{csquotes}
\usepackage{amsthm}
\usepackage{mathtools}
\usepackage{amssymb,amsfonts}
\usepackage{tikz}
\usetikzlibrary{cd}
\usepackage[backend=biber,giveninits=true,maxnames=5,style=trad-plain]{biblatex} 
\AtEveryBibitem{%
    \clearlist{language}
}

\usepackage[breaklinks,unicode=true]{hyperref}
\usepackage{cleveref} 

\DeclarePairedDelimiter{\parens}{\lparen}{\rparen}

\DeclarePairedDelimiter{\set}{\{}{\}}
\DeclarePairedDelimiter{\sqbracks}{[}{]}

\newcommand*{\toinj}{\hookrightarrow}

\newcommand*{\RR}{\mathbb{R}}

\DeclareMathOperator{\Id}{Id}
\newcommand*{\comp}{\circ} 
\newcommand*{\restr}[2]{#1|_{#2}}

\DeclareMathOperator{\supp}{supp} 

\DeclareMathOperator{\Cont}{\mathcal{C}}

\newcommand{\dd}{\mathrm{d}}
\NewDocumentCommand{\dv}{mm}{\frac{\dd #1}{\dd #2}}
\NewDocumentCommand{\pdv}{mm}{\frac{\partial #1}{\partial #2}}

\NewDocumentCommand{\orth}{om}{{#2}^{\bot\IfValueT{#1}{_#1}}}
\NewDocumentCommand{\orthR}{om}{\prescript{\bot\IfValueT{#1}{_#1}}{}{#2}}

\DeclarePairedDelimiter{\gen}{\langle}{\rangle}

\newcommand*{\lieBr}[1]{\sqbracks{#1}}
\newcommand*{\lieD}[1]{\mathcal{L}_{#1}}
\newcommand*{\lieAlg}[1]{\mathfrak{#1}}
\DeclareMathOperator{\Ad}{Ad}
\DeclareMathOperator{\im}{im}
\newcommand*{\ann}[1]{{#1}^{\circ}}

\newcommand*{\Forms}{\Omega}
\newcommand*{\VecFields}{\mathfrak{X}}

\theoremstyle{plain}
\newtheorem{thm}{Theorem}
\newtheorem*{thm*}{Theorem}
\newtheorem{lem}[thm]{Lemma}
\newtheorem*{lem*}{Lemma}
\newtheorem{prop}[thm]{Proposition}
\newtheorem*{prop*}{Proposition}
\newtheorem{cor}[thm]{Corollary}
\newtheorem*{cor*}{Corollary}

\theoremstyle{definition}
\newtheorem{defn}{Definition}
\newtheorem*{defn*}{Definition}
\newtheorem{example}{Example}
\newtheorem*{example*}{Example}

\crefname{thm}{Theorem}{Theorems}
\Crefname{thm}{Theorem}{Theorems}
\crefname{lem}{Lemma}{Lemmas}
\Crefname{lem}{Lemma}{Lemmas}
\crefname{prop}{Proposition}{Propositions}
\Crefname{Prop}{Proposition}{Propositions}
\crefname{cor}{Corollary}{Corollaries}
\Crefname{cor}{Corollary}{Corollaries}
\crefname{defn}{Definition}{Definitions}
\Crefname{defn}{Definition}{Definitions}
\crefname{example}{Example}{Examples}
\Crefname{example}{Example}{Examples}

\newtheorem*{exercise*}{Exercise}
\crefname{exercise}{exercise}{exercises}
\Crefname{exercise}{Exercise}{Exercises}  

\theoremstyle{remark}
\newtheorem{note}{Remark}
\newtheorem*{note*}{Remark}
\crefname{note}{Remark}{Remarks}
\Crefname{note}{Remark}{Remarks}

    \newcommand{\HorD}{\mathcal{H}}
    \newcommand{\VertD}{\mathcal{V}}
    \newcommand{\pHor}{\pi_\HorD}
    \newcommand{\pVert}{\pi_\VertD}
    \newcommand{\Reeb}{\mathcal{R}}
    \newcommand{\lsharp}{\sharp_\Lambda}
    \newcommand{\lbot}{\bot_\Lambda}
    \newcommand{\dhbot}{\bot_{\dd\eta}}
    \newcommand{\lorth}[1]{\orth[\Lambda]{#1}}
    \newcommand{\dhorth}[1]{{#1}^{\dhbot}}
    
    \newcommand{\clift}[1]{{#1}^c}
    \newcommand{\vlift}[1]{{#1}^v}
    \newcommand{\contr}[1]{\iota_{#1}}
    \newcommand{\Xmu}{{\pi_\mu}_* \restr{X_H}{J^{-1}(\mu)} }   
    \DeclarePairedDelimiter{\jacBr}{\lbrace}{\rbrace}

    \title{contact Hamiltonian systems}
    
   \author[M. Lainz Valcázar]{Manuel Lainz Valcázar}
   \address{Manuel Laínz:
   Instituto de Ciencias Matem\'aticas (CSIC-UAM-UC3M-UCM),
   c$\backslash$ Nicol\'as Cabrera, 13-15, Campus Cantoblanco, UAM
   28049 Madrid, Spain} \email{manuel.lainz@icmat.es}
   
   \author[M. de León]{Manuel de León}
   \address{Manuel de Le\'on: Instituto de Ciencias Matem\'aticas (CSIC-UAM-UC3M-UCM),
   c\textbackslash Nicol\'as Cabrera, 13-15, Campus Cantoblanco, UAM
   28049 Madrid, Spain \newline
   and \newline
   Real Academia de Ciencias Exactas, Físicas y Naturales, c\textbackslash de Valverde,
   22, 28004 Madrid, Spain
   } \email{mdeleon@icmat.es}
   
    \date{\today} 
    
\begin{document}

\maketitle

\begin{abstract}
    In this paper we study Hamiltonian systems on contact manifolds, which is an appropriate scenario to discuss dissipative systems. We prove a coisotropic reduction theorem similar to the one in symplectic mechanics.
\end{abstract}
\section{Introduction}
As it is well-known, symplectic geometry is the natural arena to develop Hamiltonian mechanics~\cite{Abraham1978,deLeon2011}. Indeed, given a symplectic manifold $(M, \omega)$ and a Hamiltonian function $H$ on $M$, then the
Hamiltonian vector field is provided by the equation
\begin{equation}\label{symplectic1}
i_{X_H} \, \omega = dH
\end{equation}
In Darboux coordinates $(q^i, p_i)$ we have $\omega = \dd q^i \wedge \dd p_i$ and 
\begin{equation}\label{symplectic2}
X_ H = \frac{\partial H}{\partial p_i} \frac{\partial}{\partial q^i} - 
\frac{\partial H}{\partial q^i} \frac{\partial}{\partial p_i} 
\end{equation}
so that the integral curves $(q^i(t), p_i(t))$ of $X_H$ satisfy the Hamilton equations
\begin{equation}\label{symplectic3}
\dv{q^i}{t} = \frac{\partial H}{\partial p_i} \; , \; 
\dv{p_i}{t} = - \frac{\partial H}{\partial q^i}
\end{equation}
In classical mechanics, the phase space is just the cotangent 
bundle $T^*Q$ of the configuration manifold $Q$, equipped with its canonical symplectic form $\omega_Q$.

One can also develop a time-dependent formalism using cosymplectic geometry. Indeed, a cosymplectic structure on an odd-dimensional manifold is given by a pair
$(\Omega, \eta)$ where $\Omega$ is a closed 2-form and $\eta$ is a closed 1-form such that $\eta \wedge \Omega^n \neq 0$ where $M$ has dimension $2n+1$.
Since we can obtain Darboux coordinates $(q^i, p_i, t)$ such that $\Omega = \dd q^i \wedge \dd p_i$ and $\eta = \dd t$,
we obtain the same equations than in \eqref{symplectic3} but now the Hamiltonian is time-dependent~\cite{Albert1989,Cantrijn1992,deLeon2011}.

Both scenarios produce conservative equations, so we need a different geometric structure able to produce non-conservative dynamics.

Consider now a contact manifold $(M, \eta)$ with contact form $\eta$; this means that
$\eta \wedge \dd \eta^n \neq 0$ and $M$ has odd dimension $2n+1$.
There exists a unique vector field $\mathcal R$ (called Reeb vector field) such that
$$
i_{\mathcal R} \, \dd \eta = 0 \; , \; i_{\mathcal R}\, \eta = 1
$$

There is a Darboux theorem for contact manifolds so that around each point in $M$ one can find local coordinates 
(called Darboux coordinates) $(q^i, p_i, z)$ 
such that
$$
 \eta = \dd z - p_i \, \dd q^i \; , \; 
\mathcal R = \frac{\partial}{\partial z}
$$

If we define now the vector bundle isomorphism
\begin{equation*}
    \begin{aligned}
        \flat : TM &\to T^* M  \\
        v &\to  \contr{v} \dd \eta + \eta (v)  \eta,
    \end{aligned}
\end{equation*}
then, given a Hamiltonian function $H$ on $M$ we obtain the Hamiltonian vector field $X_H$ by
$$
\flat (X_H) = \dd H - (\mathcal R (H) + H) \, \eta
$$

In Darboux coordinates we get this local expression

\begin{equation}\label{hcont2}
X_H = \frac{\partial H}{\partial p_i} \frac{\partial}{\partial q^i} - 
\parens*{\frac{\partial H}{\partial q^i} + p_i \frac{\partial H}{\partial z}}  \frac{\partial}{\partial p_i} + 
\parens*{p_i \frac{\partial H}{\partial p_i} - H} \frac{\partial}{\partial z}
\end{equation}
Therefore, an integral curve $(q^i(t), p_i(t), z(t))$ of $X_H$ satisfies the 
dissipative Hamilton equations
\begin{equation}\label{hcont3}
\dv{q^i}{t} = \frac{\partial H}{\partial p_i}, \quad
\dv{p_i}{t} = - \parens*{\frac{\partial H}{\partial q^i} +  p_i \frac{\partial H}{\partial z}}, \quad
\dv{z}{t} =  p_i \frac{\partial H}{\partial p_i} - H
\end{equation}

One can see that  \cref{symplectic3,hcont3} look very different, and the reason is the geometry used in the two formalisms.
Another clear difference is that symplectic and cosymplectic manifolds are Poisson, but a contact manifold is strictly a 
Jacobi manifold. We will discuss this fact in section~3.

The aim of this paper is to start a systematic study of contact Hamiltonian systems, that is, triples $(M, \eta, H)$ where
$(M, \eta)$ is a contact manifold and $H$ is a Hamiltonian function. Such a systems modelize thermodynamics~(both reversible~\cite{Mrugala1991}, and, more recently, irreversible~\cite{Grmela2014,Eberard2007}), statistical mechanics~\cite{Bravetti2016} as well as systems 
with dissipative forces linear in the velocities (Rayleigh dissipation), but there is a huge number of recent applications in control theory~\cite{Ramirez2017},
neurogeometry~\cite{Petitot2017} and economics~\cite{Russell2011}. A review of some these topics is available in~\cite{Bravetti2017}.

The first step in our program is just to discuss the properties of some special submanifolds of a contact manifold: isotropic, coisotropic and Legendrian
submanifolds. Our main result is just a proof of the contact version of the famous result due to A. Weinstein, the coisotropic reduction theorem~\cite{Marsden1974}.
This result provides a reduced contact quotient manifold such that a Legendre submanifold of the original contact manifold
with clean intersection with the coisotropic submanifold is projected in a reduced Legendre submanifold. This result
is used to give a simple proof of the contact reduction theorem in presence of symmetries (i.e, there is a Lie group
actiong on the contact manifold by contactomorphisms), an extension of the well-known symplectic reduction theorem proved by J.E. Marsden and A. Weinstein~\cite{Marsden1974}.
Even if the contact reduction theorem is known in the literature, we are interested in its dynamical implications
when a Hamiltonian function is also invariant by the group of symmetries.

The paper is structured as follows.  In section~2, we recall some basic definitions and results in contact geometry. Next, in section~3, we will explain how contact manifolds, along with symplectic and cosymplectic manifolds, fit in the more general framework of Jacobi manifolds. In section~4, we will define the aforementioned distinguished types of submanifolds of contact manifolds (isotropic, coisotropic and Legendrian). We will then introduce contact  Hamiltonian systems and present an interpretation as Legendrian submanifolds of the extended tangent bundle. The last two sections cover  of the coisotropic reduction theorem and the reduction theorem via the moment map.

\section{Contact manifolds}

In this section we introduce some basic definitions and results of contact geometry. Some reference textbooks are~\cite{Blair1976,Blair2002,Abraham1978,deLeon2011}.

\begin{defn}\label{def:contact_mfd}
    A \emph{contact manifold} is a pair $(M,\eta)$, where $M$ is a $(2n+1)$-dimensional manifold and $\eta\in \Forms^{1}(M)$ is a \emph{contact form}, that is, a nondegenerate  $1$-form such that $\eta \wedge {(\dd \eta)}^{n}$ is a volume form, i.e., it is non-zero at each point of $M$.
\end{defn}

Given a contact $(2n+1)$-dimensional manifold $(M, \eta)$, we can consider the following distributions on $M$, that we will call \emph{vertical} and \emph{horizontal} distribution.
    \begin{align}
        \HorD &= \ker \eta, \\
        \VertD &= \ker \dd \eta.
    \end{align}

    By the conditions on the contact form,  the following is a Whitney sum decomposition:
    \begin{equation}
        TM = \HorD \oplus \VertD,
    \end{equation}
    that is, we have the aforementioned direct sum decomposition at the tangent space of each point $x\in M$:
    \begin{equation}
        T_x M = \HorD_x \oplus \VertD_x.
    \end{equation}
    We will denote by $\pHor$ and $\pVert$ the projections on these subspaces.
    
    We notice that $\dim \HorD = 2 m$ and $\dim \VertD = 1$, and that $\restr{\dd\eta}{\HorD}$ is nondegenerate.  

The contact structure of $(M,\eta)$ gives rise to an isomorfism between tangent vectors and covectors. For each $x\in M$,
    \begin{equation}\label{def:sharp_contact}
        \begin{aligned}
            \flat: T_x M &\to     T_x^* M\\
            v            &\mapsto \contr{v}\dd \eta  + \eta(v) \eta.
        \end{aligned}
    \end{equation}
    In fact, the previous map is an isomorfism if and only if $\eta$ is a contact form~\cite{Albert1989}.
    Similarly, we obtain a vector bundle isomorfism
    \begin{equation}
        \begin{tikzcd}
            TM  \arrow[rr,"\flat"]  \arrow[dr,"\tau_M"] & & 
            T^*M \arrow[dl,"\pi_M"] \\
            & M
            &
        \end{tikzcd}
    \end{equation}
    where $\tau_M: TM \to M$ and $\pi_M: T^*M \to M$ are the canonical projections.
    
    We will also denote by $\flat:\VecFields (M) \to \Forms^1 (M)$ the corresponding isomorfism of $\Cont^{\infty}(M)$-modules of vector fields and $1$-forms over $M$. We denote $\sharp$ to the inverse of $\flat$, that is $\sharp = \flat^{-1}$

\begin{defn}\label{def:reeb}
    From the definition of the contact form and the dimensions of the vertical and horizontal distribution, we can easily proof that there exists a unique vector field $\Reeb$, named the \emph{Reeb vector field}, such that
    \begin{equation}
        \contr{\Reeb} \eta = 1, \quad \contr{\Reeb} \dd \eta = 0.
    \end{equation}

    This is equivalent to say that
    \begin{equation}
        \flat(\Reeb) = \eta,
    \end{equation}
    so that, in this sense, $\Reeb$ is the dual object of $\eta$.
\end{defn}

There are some interesting classes of maps between contact manifolds.
\begin{defn}
    A diffeomorphism between two contact manifolds $F:(M,\eta)\to (N, \xi)$ is a \emph{contactomorphism} if
    \begin{equation}
        F^*\xi = \eta.
    \end{equation}

    A diffeomorphism $F:(M,\eta)\to (N, \xi)$ is a \emph{conformal contactomorphism} if there exist a nowhere zero function $f\in \Cont^\infty(M)$ such that 
    \begin{equation}
            F^*\xi = f \eta.
    \end{equation}
    
    A vector field $X \in \VecFields M$ is a \emph{infinitesimal contactomorphism} (respectively \emph{infinitesimal conformal contactomorphism}) if its flow $\phi_t$ consists of contactomorphisms (resp. \emph{conformal contactomorphisms}).
\end{defn}

\begin{prop}
    A vector field $X$ on a contact manifold $(M,\eta)$ is an infinitesimal conformal contactomorphism if and only if
    \begin{equation}
        \lieD{X} \eta = 0.
    \end{equation}

    Furthermore $X$ is a conformal contactomorphism if and only if there exists $g \in \Cont^{\infty}(M)$ such that
    \begin{equation}
        \lieD{X} \eta = g \eta.
    \end{equation}

    In what it follows, we say that the pair $(X, g)$ is an \emph{infinitesimal conformal contactomorphism}.
\end{prop}

\begin{proof}
    Let $X \in \VecFields (M)$ and let $\phi_t$ be the corresponding flow. The proof of both statements follows from the following fact
    \begin{equation}
        \pdv{}{t} \phi_t^* \eta  =
         \phi_t^* \lieD{X} \eta.
    \end{equation}
\end{proof}

Every pair of contact manifolds are locally contactomorphic, that is, we can use a canonical set of coordinates for any contact manifold. This is implied by Darboux Theorem~\cite[Thm.~5.1.5]{Abraham1978}:

\begin{thm}[Darboux theorem]\label{def:darboux}
    Let $(M, \eta)$ be a $(2n+1)$-dimensional contact manifold. Around any point $x \in M$ there is a chart with coordinates $(x^1,\ldots,x^n,y_1\ldots,y_n,z)$ such that:
    \begin{equation*}
        \eta = \dd z -  y_i \dd x^i.
    \end{equation*}
\end{thm}

    In these coordinates,
    \begin{equation*}
        \dd\eta =  \dd x^i \wedge \dd y_i,
    \end{equation*} 
    \begin{equation}
        \Reeb = \pdv{}{z},
    \end{equation}
    and
    \begin{equation}
        \begin{aligned}
            \VertD &= \gen*{\pdv{}{z}},\\
            \HorD  &= \gen{\set{A_i,B^i}_{i=1}^n},
        \end{aligned}
    \end{equation}
    where
    \begin{align}
        A_i &= \pdv{}{x^i} - y_i \pdv{}{z},\\
        B^i &= \pdv{}{y_i}.
    \end{align}

    The local vector fields $A_i$ and $B^i$ have the following property:
    \begin{equation}
        \dd \eta(A_i,A_j)=\dd\eta(B^i,B^j)=0, \quad
        \dd \eta(A_i,B^j)= \delta^i_{j}.
    \end{equation}
    Furthermore, $\set{A_1,B^1, \ldots, A_n,B^n,\Reeb}$ and $\set{\dd x^1,\dd y_1, \ldots, \dd x^n,\dd y_n,\eta}$ are dual basis.

    This basis is not a coordinate basis of any chart, since the following Lie brackets do not vanish:
    \begin{equation}
        \lieBr{A_i,B^i} = -\Reeb, \, \forall i \in \set{1,\ldots m}.
    \end{equation}

\section{Contact manifolds and Jacobi manifolds}
It is well-known that contact manifolds are examples of a more general kind of geometric structures~\cite{Lichnerowicz1978,deLeon2017}, the so-called Jacobi manifolds, whose definition we recall below. 

\begin{defn}\label{def:jacobi_mfd}
    A \emph{Jacobi manifold} is a triple $(M,\Lambda,E)$, where $\Lambda$ is a bivector field (a skew-symmetric contravariant 2-tensor field) and $E \in \VecFields (M)$ is a vector field, so that the following identities are satisfied:
    \begin{align}
        \lieBr{\Lambda,\Lambda} &= 2 E \wedge \Lambda\\
        \lieD{E} \Lambda &= \lieBr{E,\Lambda} = 0,
    \end{align}
    where $\lieBr{\cdot,\cdot}$ is the Schouten–Nijenhuis bracket~\cite{Schouten1953,Nijenhuis1955}.
\end{defn}

The Jacobi structure $(M, \Lambda, E)$ induces a bilinear map on the space of smooth functions. We define the \emph{Jacobi bracket}:
    \begin{equation}
        \begin{aligned}
            \jacBr{\cdot,\cdot}: \Cont^\infty(M) \times \Cont^{\infty}(M) & \to \RR, \\
            (f,g) &\mapsto \jacBr{f,g},
        \end{aligned}
    \end{equation}
    where
    \begin{equation}
        \jacBr{f,g} = \Lambda(\dd f, \dd g) + f E(g) - g E (f).
    \end{equation}

    This bracket is bilinear, antisymmetric, and satisfies the Jacobi identity. Furthermore it fulfills the weak Leibniz rule:
    \begin{equation}
        \supp(\jacBr{f,g}) \subseteq \supp (f) \cap \supp (g).
    \end{equation}
    That is, $(\Cont^\infty(M), \jacBr{\cdot,\cdot})$ is a local Lie algebra in the sense of Kirillov. Conversely, given a local Lie algebra $\Cont^{\infty}(M)$, we can find a Jacobi structure on $M$ such that the Jacobi bracket coincides with the algebra bracket (see~\cite{Kirillov1976,Lichnerowicz1978}).

Given a contact manifold $(M,\eta)$ we can define a Jacobi structure $(M, \Lambda, E)$ by taking
\begin{equation}
    \Lambda(\alpha,\beta) = - \dd \eta (\sharp\alpha, \sharp\beta), \quad
    E = - \Reeb,
\end{equation}
where $\sharp$ is defined as in \cref{def:sharp_contact}. Indeed, a simple computation shows that $\Lambda$ and $E$ satisfy the conditions of \cref{def:jacobi_mfd}.

One important particular case of Jacobi manifolds are Poisson manifolds, such as symplectic manifolds. A Poisson manifold is a manifold $M$ equipped with a Lie bracket $\jacBr{\cdot,\cdot}$ on $\Cont^\infty(M)$ that satisfies the following Leibniz rule
    \begin{equation}
        \jacBr{fg,h} = f \jacBr{g,h} + \jacBr{f,h} g.
    \end{equation}
This can be seen to imply the weak Leibniz rule, giving a local Lie algebra structure on $\Cont^\infty(M)$. In terms of the Jacobi structure $(M, \Lambda, E)$, a simple computation shows that the Jacobi brackets are Poisson if and only if $E=0$, hence a Poisson manifold will be denoted $(M,\Lambda)$. Another kind of Poisson manifolds are cosymplectic manifolds.

\begin{example}[Cosymplectic manifold]
    A cosymplectic manifold~\cite{Cantrijn1992,Cappelletti-Montano2013} is given by a triple $(M, \Omega, \eta)$ where $M$ is a $(2n+1)$-dimensional $\Omega$ a closed $2$-form and $\eta$ is a closed $1$-form. 
    
    We consider the isomorfism
    \begin{equation}
        \begin{aligned}
            \flat: T M &\to     T_x^* M\\
            X            &\mapsto \contr{X}\Omega + \eta(X) \eta.
        \end{aligned}
    \end{equation}
    If we denote its inverse by $\sharp = \flat^{-1}$, then
    \begin{equation*}
        \Lambda(\alpha, \beta) =
        \Omega(\sharp\alpha,\sharp\beta),
    \end{equation*}
    is a Poisson tensor on $M$.
\end{example}

We note that contact manifolds are not Poisson, since $E=-\Reeb \neq 0$. Other important examples of non-Poisson Jacobi manifolds are locally conformally symplectic manifolds.

\begin{example}[Locally conformal symplectic manifolds]
Let $(M,\Omega)$ be an \emph{almost symplectic manifold}. That is, a manifold $M$  equipped with a nondegenerate and antisymmetric, but not necessarily closed two-form $\Omega\in \Forms^2(M)$.

$(M, \Omega)$ is said to be \emph{locally conformally symplectic} if for each point $x\in M$
there is an open neighborhood $U$ such that $d(e^{\sigma}\Omega)=0,$ for some $\sigma \in \Cont^{\infty}(U)$, so $(U,e^{\sigma}\Omega)$
is a symplectic manifold. If $U=M$, then it is said to be globally conformally symplectic.
An almost symplectic manifold is a locally (globally) conformally symplectic if there exists a one-form
$\gamma$ that is closed $d\gamma=0$ and
\begin{equation*}
 d\Omega=\gamma\wedge \Omega.
\end{equation*}
The one-form $\gamma$ is called the \emph{Lee one-form}. Locally conformally symplectic manifolds with Lee form $\gamma=0$ are symplectic manifolds.
We define a bivector $\Lambda$ on $M$ and a vector field $E$ given by
\begin{equation*}
 \Lambda(\alpha,\beta)=\Omega(\flat^{-1}(\alpha),\flat^{-1}(\beta))=\Omega(\sharp(\alpha),\sharp(\beta)),\quad E=\flat^{-1}(\gamma),
\end{equation*}
with $\alpha,\beta\in \Omega^{1}(M)$ and $\flat: \mathfrak{X}(M)\rightarrow \Omega^{1}(M)$ is the isomorphism of $C^{\infty}(M)$
modules defined by $\flat(X)=\iota_X\Omega$. Here $\sharp=\flat^{-1}$. In this case, we also have $\sharp_{\Lambda}=\sharp$. The vector field $Z$ satisfies $\iota_{E}\gamma=0$ and $\mathcal{L}_{E}\Omega=0, \mathcal{L}_{E}\gamma=0$.
Then, $(M,\Lambda,E)$ is an even dimensional Jacobi manifold. 

\end{example}

The Jacobi structure also induces a morphism between covectors and vectors.
\begin{defn}\label{def:sharp_jacobi}
    Let $(M,\Lambda,E)$ be a Jacobi manifold. We define the following morphism of vector bundles:
    \begin{equation}
        \begin{aligned}
            \lsharp: T M^* &\to     T M\\
            \alpha &\mapsto \Lambda(\alpha, \cdot ),
        \end{aligned}
    \end{equation}
    which also induces a morphism of $\Cont^{\infty}(M)$-modules between the covector and vector fields, as in \cref{def:sharp_contact}.

    In the case of a contact manifold, this is given by
    \begin{equation}
        \lsharp (\alpha) =
        \sharp (\alpha) - \alpha(\Reeb) \Reeb,
    \end{equation}
    where the equality follows from this computation:
\begin{equation}
    \begin{aligned}
        \Lambda(\alpha, \beta) &= 
        - \contr{\sharp\beta} \contr{\sharp \alpha} \dd \eta \\ &=
          \contr{\sharp\alpha}\contr{\sharp\beta} \dd \eta \\ &=
          \contr{\sharp\alpha} (\beta - \eta(\sharp\beta)\eta) =
          \beta(\sharp\alpha) - \alpha(\Reeb) \beta(\Reeb),
    \end{aligned}
\end{equation}
where we have used that
\begin{equation}
    \beta = \flat\sharp \beta = 
    \contr{\sharp\beta} \dd \eta + \eta(\sharp \beta) \eta =
    \contr{\sharp\beta} \dd \eta + \beta(\Reeb) \eta. 
\end{equation}

\begin{note}\label{rem:lsharp_contact}
    For a contact manifold, $\lsharp$ is not an isomorfism. In fact, $\ker\lsharp=\gen{\eta}$ and $\im \lsharp = \HorD$.    
\end{note}
\end{defn}

We will end this section by stating the Structure Theorem for Jacobi manifolds~\cite{Dazord1991}, after introducing some terminology.

Vector fields associated with functions $f$ on the algebra of smooth functions $C^{\infty}(M)$ are defined as
\begin{equation} \label{def:Ham_vf_jacobi}
    X_f=\sharp_{\Lambda}(\dd f) + f E,
\end{equation}

The \emph{characteristic distribution} $\mathcal{C}$ of $(M,\Lambda,E)$ is generated by the
values of all the vector fields $X_f$:
\begin{align}
    \mathcal{C}_p &= \gen{\set{X_f(p) \mid f \in \Cont^\infty(M), p \in M}},\\
    \mathcal{C} &= \bigsqcup_{p\in M} \mathcal{C}_p.
\end{align}

This characteristic distribution $\mathcal{C}$ is defined in terms of $\Lambda$ and $E$ as follows
\begin{equation*}
 \mathcal{C}_p={(\lsharp)}_p(T^{*}_pM)+<E_p>,\quad \forall p\in M
\end{equation*}
where ${(\lsharp)}_p:T_p^{*}M\rightarrow T_pM$ is the restriction of $\sharp_{\Lambda}$ to $T^{*}_pM$ for every $p\in M$.
Then, $\mathcal{C}_p=\mathcal{C}\cap T_{p}M$ is the vector subspace of $T_pM$ generated by
$E_p$ and the image of the linear mapping $\sharp_p$. 

The distribution is said to be \emph{transitive} if the characteristic distribution
is the whole tangent bundle $TM$.

\begin{thm}[Structure theorem for Jacobi manifolds]
    The characteristic distribution of a Jacobi manifold $(M,\Lambda,E)$ is completely integrable in the sense of Stefan--Sussmann, thus
    $M$ defines a foliation whose leaves are not necessarily of the same dimension, and it is called the \emph{characteristic foliation}. Each leaf has a unique
    transitive Jacobi structure such that its canonical injection into $M$ is a Jacobi map (that is, it preserves the Jacobi brackets).
    Each can be
    \begin{enumerate}
     \item A locally conformally symplectic manifold (including symplectic manifolds) if the dimension is even.
    \item A manifold equipped with a contact one-form if its dimension is odd.
    \end{enumerate}    
\end{thm}

\begin{note}
    A \emph{completely integrable distribution in the sense of Stefan--Sussmann} is involutive but not necessarily of constant rank, therefore it defines a singular foliation in which leaves are allowed to have different dimensions. A rigorous and more complete statement of this result, which is a generalization of Frobenius theorem can be read in \cite[Thm.~4.2]{Sussmann1973}.
\end{note}

\section{Submanifolds of a contact manifold}

As in the case of symplectic manifolds, we can consider several interesting types of submanifolds of a contact manifold $(M,\eta)$. To define them, we will use the following notion of \emph{complement} given by the contact structure:

\begin{defn}[Contact complement]
    Let $(M,\eta)$ be a contact manifold and $x\in M$. Let $\Delta_x\subset T_x M$ be a linear subspace. We define the \emph{contact complement} of $\Delta_x$
    \begin{equation}
        \lorth{\Delta_x} = \lsharp(\ann{\Delta_x}),
    \end{equation}
    where 
    $\ann{\Delta_x}= \set{\alpha_{x}\in T_x^*M \mid \alpha_x(\Delta_x)=0}$ is the annihilator.

    We extend this definition for distributions $\Delta\subseteq TM$ by taking the complement pointwise in each tangent space.
\end{defn}

\begin{defn}
    Let $N\subseteq M$ be a submanifold. We say that $N$ is:
    \begin{itemize}
        \item \emph{Isotropic} if $TN\subseteq\lorth{TN}$.
        \item \emph{Cosotropic} if $TN\supseteq\lorth{TN}$.
        \item \emph{Legendrian} if $TN=\lorth{TN}$.
    \end{itemize}
\end{defn}

Legendrian submanifolds play an important role in contact geometry, similar to Lagrangian submanifolds in symplectic geometry. We will latter present a dynamical interpretation in~\cref{thm:image_legendrian}. On thermodynamics applications, they represent equilibrium states of the system~\cite{Mrugala1991}.

The coisotropic condition can be written in local coordinates as follows.
\begin{prop}    
    Let $N\subseteq M$ be a $k$-dimensional manifold given locally by the zero set of functions $\phi_a:U\to \RR$, with $a\in \set{1, \ldots, k}$. We use Darboux coordinates (\cref{def:darboux}).
    We have that
    \begin{equation*}
        \lorth{TN} = \gen{\set{Z_a}_{a=1}^k},
    \end{equation*}
    where,
    \begin{equation*}
        \begin{aligned}
            Z_a = \lsharp(\dd \phi_a) &=
             A_i(\phi_a)B^i - B^i(\phi_a) A_i \\ &=
              \parens*{\pdv{\phi_a}{x^i} + y_i \pdv{\phi_a}{t}}\pdv{}{y_i}+ 
            \pdv{\phi_a}{y_i} \parens*{\pdv{}{x^i} -y_i\pdv{}{z}}.
        \end{aligned}
    \end{equation*}    
    
    Therefore, $N$ is coisotropic if and only if, $Z_a(f_b)=0$ for all  $a,b$. In coordinates:
    \begin{align}
        A_i(\phi_a)B^i(\phi_b) - B^i(\phi_a) A_i(\phi_b)&=0\\
        \parens*{\pdv{\phi_a}{x^i} + y_i \pdv{\phi_a}{z}}\pdv{\phi_v}{y_i}+ 
        \pdv{\phi_a}{y_i} \parens*{\pdv{\phi_b}{x^i} -y_i\pdv{\phi_b}{z}} &= 0.
    \end{align}
\end{prop}

For studying the properties of these submanifolds we need to analyze the orthogonal complement $\lbot$ with some detail.

\begin{prop}\label{thm:orth}
    Let $\Delta, \Gamma \subseteq TM$ be distributions. The contact complement has the following properties:
    \begin{itemize}
        \item $\lorth{(\Delta \cap \Gamma)} = \lorth{\Delta} + \lorth{\Gamma}$.
        \item $\lorth{(\Delta + \Gamma)} = \lorth{\Delta} \cap \lorth{\Gamma}$.
    \end{itemize}
\end{prop}

\begin{proof}
    This is due to the fact that the annihilator interchanges intersections and sums, while the linear map $\lsharp$ preserves them. 
\end{proof}

We note that the horizontal distribution $(\HorD, \dd\eta)$ is symplectic. Let $\Delta\subseteq H$. We denote by $\dhbot$ the symplectic orthogonal component:
\begin{equation}
    \dhorth{\Delta}=\set{v \in TM \mid \dd\eta(v,\Delta)=0},
\end{equation}

We remark that $\Reeb\in \dhorth{\Delta}$ for any distribution $\Delta$. There is a simple relationship between both notions of orthogonal complement:
\begin{prop}\label{thm:orth_rel}
    Let $\Delta \subseteq TM$ be a distribution, then 
    \begin{equation}
        \lorth{\Delta} = \dhorth{\Delta} \cap \HorD.
    \end{equation}
\end{prop}

\begin{proof}
    Let $v \in \dhorth{\Delta} \cap \HorD$, that is, $\dd\eta(v,\Delta) = 0$ and $v$ is horizontal. We will see that $v \in \lorth{\Delta}$. Indeed, we can easily check that
    \begin{equation}
        \lsharp(\contr{v} \dd \eta) =
        \flat^{-1}(\contr{v} \dd \eta) - \dd \eta(v, \Reeb) =  \flat^{-1}(\contr{v} \dd \eta) = v,
    \end{equation}
    since,
    \begin{equation}
        \flat (v) = \contr{v} \dd\eta + \eta(v) \eta = \contr{v} \dd\eta,
    \end{equation}
    because $v$ is horizontal. Therefore $\dhorth{\Delta} \cap \HorD \subseteq \lorth{\Delta}$. 
    
    To prove the other inclusion, we just count the dimensions. Let $k=\dim \Delta$, so that $\dim \ann{\Delta}=2n+1-k$. Since $\lorth{\Delta} = \lsharp(\ann{\Delta})$, and $\ker(\lsharp)=\gen{\Reeb}$, we find out that if $\Reeb \in \ann{\Delta}$ (i.e., $\Delta$ is horizontal), and then 
    $\dim\lsharp(\ann{\Delta}) = 2n-k$. Otherwise,  $\dim\lsharp(\ann{\Delta}) = 2n-k+1$. This trivially coincides with the dimension of the right hand side.
\end{proof}

We have the following possibilities regarding the relative position of a distribution $\Delta$ in a contact manifold and the vertical and horizontal distributions, 
\begin{defn}\label{thm:orth_casos}
    Let $\Delta\subseteq TM$ be a rank $k$ distribution. We say that a point $x\in M$ is
    \begin{enumerate}
        \item \emph{Horizontal} if $\Delta_x = \Delta_x \cap \HorD_x$.
        \item \emph{Vertical} if $\Delta_x = (\Delta_x \cap \HorD_x) \oplus \gen{\Reeb_x}$.
        \item \emph{Oblique} if $\Delta_x =(\Delta_x \cap \HorD_x)\oplus \gen{\Reeb_x + v_x}$, with $v_x \in \HorD_x \setminus \Delta_x$.
    \end{enumerate}
    If $x$ is horizontal, then $\dim \lorth{\Delta} = 2n - k$. Otherwise, $\dim \lorth{\Delta}= 2n + 1 -k$. 
\end{defn}

There is a characterization of the concept of isotropic/Legendrian submanifolds as integral submanifolds of $\eta$ that we will state in \cref{prop:contact_integral}.

\begin{prop}\label{prop:contact_integral}
    A submanifold $N$ of a $(2n+1)$-dimensional contact manifold $(M, \eta)$ is isotropic if and only if $\restr{\eta}{TN} = 0$. Furthermore, it is Legendrian if and only if it is isotropic and $\dim(N)=m$.
\end{prop}
\begin{proof}
    $N$ is isotropic if and only if, in each $x\in N$,
    \begin{equation}
        \dd \eta(v,w) = \eta(v)\eta(w),
    \end{equation}
    for every $v,w \in T_xN$. Since the left hand side of the equation is symmetric and the right hand side is antisymmetric, this is equivalent to the fact that $\eta(v)=0$ for any $v \in T_x N$.

    If $N$ is Legendrian, then it consists of horizontal points, since the image of $\lsharp$ is horizontal (\cref{rem:lsharp_contact}). The claim follows from counting the dimensions using \cref{thm:orth_casos}.
\end{proof}

\section{Contact Hamiltonian systems}
In this section we will study some properties of the Hamiltonian vector fields on contact manifolds. 
\begin{defn}
    Given a smooth real function $H$ on a contact manifold $(M,\eta)$, we define its \emph{Hamiltonian vector field} as
    \begin{equation}
        X_H = \lsharp (dH) - H\Reeb,
    \end{equation}
    or equivalently,
    \begin{equation}
        \flat (X_H) =  dH - (\Reeb(H) + H) \eta. 
    \end{equation}
    In Darboux coordinates, this is written as follows
    \begin{equation}
        X_H =  \frac{\partial H}{\partial y_i} \frac{\partial}{\partial x^i}
        - \parens*{\frac{\partial H}{\partial x^i} + y_i \frac{\partial H}{\partial z}} \frac{\partial}{\partial y_i} + 
         \parens*{ y_i \frac{\partial H}{\partial {y_i }} - H} \frac{\partial}{\partial z}.
    \end{equation}
\end{defn}

An integral curve of this vector field satisfies the dissipative Hamiltonian equations:
\begin{align}
    \dot{x}^i &= \frac{\partial H}{\partial y_i},\\
    \dot{y}_i &= -\parens*{\frac{\partial H}{\partial x^i} + y_i \frac{\partial H}{\partial z}},\\
    \dot{z} &= { y_i \frac{\partial H}{\partial {y_i }} - H}.
\end{align}

This equations are a generalization of the conservative Hamilton equations. We recover this particular case when $\Reeb(H)=0$.

The contact Hamiltonian vector fields model the dynamics of dissipative systems. As opposed to the case of symplectic Hamiltonian systems, the evolution does not preserve the energy or the natural volume form.

\begin{thm}[Energy dissipation]
    Let $(M,\eta,H)$ be a Hamiltonian system. The flow of the Hamiltonian vector field $X_H$ does not preserve the energy $H$. In fact
    \begin{equation}
        \lieD{X_H} H = - \Reeb(H) H.
    \end{equation}
    
    As well, the contact volume element $\Omega = \eta\wedge(\dd \eta)^n$ is not preserved. 
    \begin{equation}
        \lieD{X_H} \Omega=
         - (n+1) \Reeb (H) \Omega.
    \end{equation}

    However, if $H$ and $\Reeb(H)$ are nowhere zero, there is a unique volume form depending on the Hamiltonian\footnote{This is useful for applications in statistical mechanics, as can be read on the article~\cite{Bravetti2015}. There might be other invariant volume forms.} (up to multiplication by a constant) that is preserved~\cite{Bravetti2015}. By this, we mean that there is a unique form $\tilde \Omega = (g \comp H)\Omega$, where $g:g(H)\to \RR$ is a smooth function, which is given by
    \begin{equation}
        \tilde\Omega = {H}^{-(n+1)}  \Omega.
    \end{equation}
\end{thm}
\begin{proof}
    The first claim follows from the definition of $X_H$.

    We proceed with the second claim. A straightforward computation using Cartan's formula shows that
    \begin{equation}
        \lieD{X_H} \eta = \contr{X_H} \dd \eta + \dd H = - \Reeb(H) \eta,
    \end{equation}
    where
    \begin{equation}
        \dd H = \flat(\sharp (\dd H)) = 
        \contr{\sharp\dd H} \dd \eta + \eta(\sharp \dd H) \eta =
        \contr{X_H} \dd \eta + \Reeb(H) \eta.     
    \end{equation}

    Now we can compute the derivative using the product rule
    \begin{equation}
        \begin{aligned}
            \lieD{X_H}( \eta \wedge (\dd \eta)^n ) &= 
            - \Reeb(H) \eta \wedge (\dd \eta)^n \\ &+ 
            n \eta \wedge (\dd \eta)^{n-1} \wedge(-(\dd \Reeb(H)) \eta - \Reeb(H) \dd \eta ) \\ &= 
            -(n+1) \Reeb(H) \eta \wedge (\dd \eta)^n,
        \end{aligned}
    \end{equation}
    as we wanted to show.

    Last of all, consider a volume form $\tilde \Omega = (g \comp H)\Omega$. Then
    \begin{equation}
        \begin{aligned}
            \lieD{X_H}{\tilde \Omega} &= 
            \lieD{X_H}{(g \comp H)} \Omega +  (g \comp H) \lieD{X_H}{\Omega} \\&=
            -((g' \comp H) H \Reeb(H) + (n+1)(g \comp H) \Reeb(H))\Omega.
        \end{aligned}
        \end{equation}
    Since $\Omega$ is a volume form and $\Reeb(H)$ is non-zero, this Lie derivative vanishes if and only if
    \begin{equation}
        g'(h) h + (n+1)g(h) = 0.
    \end{equation}
    Hence, $g$ is the solution to this linear ODE, which is unique up to multiplication by a constant and it is given by
    \begin{equation}
        g(h) = C h^{-(n+1)},
    \end{equation}
    where $C\in \RR$ is a constant. Therefore $\tilde\Omega$ is preserved if and only if it is of the form
    \begin{equation}
        \tilde\Omega = C {H}^{-(n+1)}  \Omega.
    \end{equation}
\end{proof}

\begin{prop}\label{thm:hamiltonian_bijection}
    Given a $(2n+1)$-dimensional contact manifold $(M,\eta)$, the map $H \to X_H$ is a Lie algebra isomorfism between the set of smooth functions with the Jacobi bracket and the set of infinitesimal conformal contactomorphisms with the Lie bracket. Its inverse is given by $X \to - \contr{X} \eta$.

    Furthermore, $X_H$ is an infinitesimal contactomorphism if and only if $\Reeb(H) = 0$.
\end{prop}

\begin{proof}
    Let $H\in\Cont^\infty(M)$. We will see that the map is well-defined. 
    Since $\lieD{X_H} \eta = - \Reeb(H) \eta$, $X_H$ is an infinitesimal conformal contactomorphism, and it is an infinitesimal contactomorphism if and only if $\Reeb (H) = 0$.

    By contracting $X_H$ with the contact form, we can recover $H$, so the Hamiltonian map is a bijection. 
    \begin{equation}
        -\eta(X_H)= -\eta(\lsharp{dH}) + \eta(H\Reeb) = H.
    \end{equation}

    Last of all, we will show that $X \to - \contr{X} \eta$ is an antihomomorphism. That is, if $F,G \in \Cont^{\infty}(M)$,  
    \begin{equation}
        \contr{\lieBr{X_F,X_G}} \eta = \jacBr{F,G}.
    \end{equation}
    For this, we will use again Cartan's formula. We first notice that $\dd \eta (X_F,X_G)= - \Lambda(\dd F, \dd G)$. Indeed,
    \begin{equation}
        \begin{aligned}
            \dd & \eta{(X_F,X_G})= 
            X_F(\eta(X_G)) -  X_G (\eta(X_F)) -\contr{\lieBr{X_F,X_G}} \eta \\ &= 
            - X_F(G)  + X_G(F)   -\contr{\lieBr{X_F,X_G}} \eta \\ &=
            - \lsharp(\dd F)(G) + F \Reeb (G) + \lsharp(\dd G)(F) - G \Reeb(F) -\contr{\lieBr{X_F,X_G}} \eta \\ &= 
            -2 \Lambda(\dd F, \dd G)+
             F\Reeb(G) - G \Reeb(F) -\contr{\lieBr{X_F,X_G}} \eta,
        \end{aligned}
    \end{equation}
    since $\lsharp(\dd F)(G) = -\lsharp(\dd G)(F) = \Lambda(\dd F, \dd G)$, due to the antisymmetry of $\Lambda$. From this, we get
    \begin{equation}
        -\jacBr{F,G}= - 2 \jacBr{F,G} - \contr{\lieBr{X_F,X_G}} \eta,
    \end{equation}
    hence $\jacBr{F,G} = \contr{\lieBr{X_F,X_G}} \eta$.
\end{proof}

\begin{defn}
    A \emph{contact Hamiltonian} system is a triple $(M,\eta,H)$, where $(M,\eta)$ is a contact manifold and $H$ is a smooth real function on $M$.
\end{defn}

In symplectic geometry, the image of a vector field is a Lagrangian submanifold of the tangent bundle, with the appropriate symplectic structure~\cite{Tulczyjew1976a}. Motivated by this and a similar result in \cite[Prop.~3]{Cantrijn1992} for cosymplectic manifolds, we may ask if there is a similar relationship between Hamiltonian vector fields and Legendrian submanifolds. The following two theorems \cite{Ibanez1997} will lead us to an affirmative answer.

\begin{prop}
    Let $(M,\eta)$ be a contact manifold. Let $\bar{\eta}$ be a one form on $TM \times \RR$ such that
    \begin{equation}
        \bar\eta = \clift{\eta} + t \vlift{\eta},
    \end{equation}
    where $t$ is the usual coordinate on $\RR$ and $\clift{\eta}$ and $\vlift{\eta}$ are the complete and vertical lifts~\cite{Yano1973} of $\eta$ to $TM$.
    
    Then, $(TM \times \RR, \bar \eta)$ is a contact manifold with Reeb vector field $\bar\Reeb = \vlift{\Reeb}$. 
\end{prop}
\begin{proof}
    We denote by $\bar\flat$ the $\Cont^{\infty}(TM \times \RR)$-module morphism given by
    \begin{equation}
        \begin{aligned}
            \bar \eta : \VecFields(M) &\to \Forms^1(M)\\
            X &\mapsto \contr{X} \dd \bar\eta + \bar\eta(X) \bar{\eta}.
        \end{aligned}
    \end{equation}
    The map $\flat$ denotes the contact isomorfism of $(M,\eta)$ (see \cref{def:sharp_contact}). Let $X$ be such that $\eta(X)=0$ and let $t$ be the coordinate corresponding to $\RR$ in $M \times \RR$. Then it follows from a straightforward computation that
    \begin{equation}
        \begin{aligned}
            \bar\flat(\vlift{X}) &= \vlift{\flat(\bar{X})},\\
            \bar\flat(\clift{X}) &= \clift{\flat(\bar{X})} + t \vlift{\flat(\bar{X})},\\
            \bar\flat(\vlift{\Reeb}) &= \bar\eta,\\
            \bar\flat(\clift{\Reeb}) &= -\dd t + t\bar\eta,\\
            \bar\flat\parens*{\pdv{}{s}} &= \vlift{\eta}.
        \end{aligned}
    \end{equation}
    Hence, $\bar\flat$ is an isomorfism and $\bar\eta$ is a contact form (we recall that vertical and complete lifts are linearly independent).
\end{proof}

\begin{thm}\label{thm:contactomorphism_legendrian}
    Let $(M,\eta)$ be a contact manifold, and let $X\in\VecFields(M)$, $f\in \Cont^\infty (M)$. We denote
    \begin{equation}
        \begin{aligned}
            X \times f: M &\to TM \times \RR\\
            p &\mapsto (X_p, f(p)),
        \end{aligned}
    \end{equation}
    Then $(X, - f)$ is a conformal Jacobi infinitesimal transformation if and only if $\im(X \times f) \subseteq (TM \times \RR, \bar{\eta})$ is a Legendrian submanifold.
\end{thm}
\begin{proof}
    Let $L=(X \times f)(M)$. Take $x\in M$ and $z = (X_x,f(x)) \in L$. Then
    \begin{equation}
        \begin{aligned}
            T_z L = \set{w_v = \parens{{(X_x)}_* (v), v(f)} \mid v \in T_x M}
            &\subseteq T_z(M \times \RR) \\&\simeq
             T_{X_x} M \oplus T_{f(x)} \RR.
        \end{aligned}
    \end{equation}
    On the other hand, by the properties of lifts \cite{Yano1973},
    \begin{equation}
        X^*(\clift{\eta}) = \lieD{X} \eta, \quad 
        \vlift{(\eta_{X_x})} \comp {(X_x)}_* = \eta_x.
    \end{equation}
    Notice that $L$ is $n$-dimensional, hence, by \cref{prop:contact_integral}, it is Legendrian if and only if $\eta$ vanishes on $TL$. Let $v \in T M$,
    \begin{equation}
        \bar \eta(w_v) = (\lieD{X} \eta)_x (v) + f(x) \eta_x(v).
    \end{equation}
    Hence, $\bar\eta$ vanishes on $TL$ precisely when $(X,-f)$ is an infinitesimal conformal contactomorphism. 
\end{proof}

This result states that the image of vector field $X_H$, suitably included in the contactified tangent bundle, is a Legendrian submanifold. Roughly speaking, Hamiltonian vector fields are particular cases of Legendrian submanifolds.
\begin{thm}\label{thm:image_legendrian}
    Let $(M, \eta)$ be a contact manifold system. Let $X \in \VecFields(M)$. We define the extended vector field $\bar X$ as
    \begin{equation}
        \bar X = X \times (- \Reeb(\eta(X))) : M \to TM \times \RR.
    \end{equation}
    Then, $\im (\bar X) \subseteq (TM \times \RR, \bar \eta)$ is a Legendrian submanifold if and only if $X$ is a Hamiltonian vector field.
\end{thm}
\begin{proof}
    Let $X=X_H$ be a Hamiltonian vector field. Then, by \cref{thm:hamiltonian_bijection}, $\eta(X)=-H$, and $(X_H, -\Reeb(H))$ is a infinitesimal conformal contactomorphism, so the image of $\bar X = X \times (\Reeb(H))$ is a Legendrian submanifold by \cref{thm:image_legendrian}.

    Conversely, by the same theorems, if $X$ is not Hamiltonian, then it is not an infinitesimal conformal contactomorphism. Hence, $\im \bar X$ cannot be a Legendrian submanifold.
\end{proof}

\section{Coisotropic reduction in contact geometry}
In this section we present a result of reduction in the context of contact geometry, which is analogous to the well-known coisotropic reduction in symplectic geometry. This theorem is not true in more general contexts, such as Poisson or Jacobi manifolds, where more structure is needed to perform the reduction~\cite{Marsden1986}.

\begin{defn}
    Given a coisotropic submanifold $\iota:N\toinj M$, we define
    \begin{equation*}
        \begin{aligned}
            \eta_0 &= \iota^* \eta = \restr{\eta}{TN},\\
            \dd \eta_0 &= \iota^* (\dd \eta) = \dd {(\iota^*\eta)}.\\
        \end{aligned}
    \end{equation*}

    We call \emph{characteristic distribution} of $N$ to 
    \begin{equation*}
        \lorth{TN} =
        \ker(\eta_0) \cap \ker(\dd \eta_0).
    \end{equation*}
\end{defn}
\begin{proof}
    We shall proof the last equality: 
    \begin{align*}
        \lorth{TN} &= 
        \set{v \in TM \cap \HorD \mid \dd\eta(v,TN) = 0} \\ &=
        \set{v \in TN \mid \eta(v) = \eta_0(v)= 0, \dd\eta(v,TN) = \dd\eta_0(v,TN) =0} \\&=
        \ker \eta_0\cap \ker \dd\eta_0,
    \end{align*}
    where the first equality is due to \cref{thm:orth_rel} and the second one to the fact that $N$ is coisotropic, which ensures that all orthogonal vectors are in $TN$.
\end{proof}

\begin{thm}[Coisotropic reduction in contact manifolds]\label{thm:coisotropic_reduction}
    Let $\iota:N\toinj M$ be a coisotropic submanifold. Then $\lorth{TN}$ is an involutive distribution. 
    
    Assume that the quotient $\tilde N=TN/\lorth{TN}$ is a manifold and that $N$ does not have horizontal points. Let $\pi: N \to \tilde N$ be the projection. Then there is a unique $1$-form $\tilde{\eta}\in \Forms^1(\tilde N)$ such that
    \begin{equation}
        \pi^* \tilde \eta = \iota^* \eta.
    \end{equation}
    Moreover, $(N, \tilde \eta)$ is a contact manifold.
    
    Furthermore, if $N$ consists of vertical points, $\tilde \Reeb = \pi_*{\Reeb}$ is well defined and is the corresponding Reeb vector field.
\end{thm}
\begin{proof}
    First of all, we will proof that the distribution $\lorth{TN}$ is involutive. Let $X,Y\in \lorth{TN} =  \ker(\eta_0) \cap \ker(\dd \eta_0)$. We will show that  $\lieBr{X,Y}\in \lorth{TN}$. By Cartan's formula:
    \begin{equation}
        0 =\dd \dd \eta_0(X,Y) = \eta_0(X) - \eta_0(Y) - \eta_0(\lieBr{X,Y})= - \eta_0(\lieBr{X,Y}),
    \end{equation}
thus $\lieBr{X,Y}\in \ker(\eta_0)$. Now, we will use Cartan's formula with $\dd \eta_0$. Let $Z\in TN$,
    \begin{equation}
        \begin{gathered}
            0 = \dd \eta_0(X,Y,Z) = 
            X(\dd \eta_0(Y,Z)) - Y(\dd \eta_0(X,Z)) + Z(\dd \eta_0(X,Y))\\
            - \dd \eta_0(\lieBr{X,Y},Z) + \dd \eta_0(\lieBr{X,Z},Y) - \dd \eta_0(\lieBr{Y,Z},X) = 
            - \dd \eta_0(\lieBr{X,Y},Z),
        \end{gathered}
    \end{equation}
    from which we conclude that $\lieBr{X,Y}\in \lorth{TN}$.

    Now we will check that there is a unique $1$-form $\tilde \eta$ such that $\pi^* \tilde{\eta} = \iota^* \eta$. We note that it is enough to show this locally (on open subsets of $\tilde N$).

    For proving the existence, we take a smooth section $X: \tilde{N} \to N$ of $\pi$ (that is $X \comp \pi = \Id_{\tilde{M}}$), which always exists locally because $\pi$ is a submersion. We can let $\tilde \eta =  X^* \eta_0$. 
    
    We check the uniqueness in the tangent space of each $x \in TN$. We know that $\ker(\eta_0)_x \supseteq \lorth{TN}_x = {\ker (T \pi)}_x$. Thus, $\tilde \eta_x$ does not depend on the chosen element of the preimage of $T_x \pi$. The following diagram illustrates this situation.
    \begin{equation}
        \begin{tikzcd}
            T_x N  \arrow[r,"T_x \iota"]  \arrow[d,"T_x \pi"] \arrow[dr, "{(\eta_0)}_x"] &
            T_xM  \arrow[d,"\eta_{x}"] \\ 
            T_{[x]} \tilde{N} \arrow[dashed,r,"\tilde{\eta}_{[x]}"] &
            \RR.
        \end{tikzcd}
    \end{equation}

    We also have to prove that this projection does not depend on the base point of the fiber $\pi^{-1}(\set{p})\supseteq N$. We compute the Lie derivative of $\eta_0$ in the direction $X\in \lorth{TN}$ using, again, Cartan's formula:
    \begin{equation*}
        \lieD{X} \eta_0 = \dd \contr{X} \eta_0  +\contr{X} \dd \eta_0 =0,
    \end{equation*}
    hence $\tilde \eta$ is well-defined. Likewise we can check that $\Reeb$ projects to $\tilde \Reeb$ on the vertical points,
    
    On the horizontal points, we obtain $\tilde \eta = 0$, thus we do not get a contact form.

    In non-horizontal points, $\tilde\eta$ is nondegenerate because $\tilde\eta(\pi_*(\Reeb+v))=1$. Given that we have taken the quotient by $\ker(\dd\eta_0)\cap \ker(\eta_0)$, $\dd\tilde\eta$ is obviously nondegenerate.
\end{proof}

\begin{cor}
    With the notations from previous theorem, assume that $L\subseteq M$ is Legendrian, $N$ does not have horizontal points, and $N$ and $L$ have clean intersection (that is,  $N \cap L$ is a submanifold and $T(N\cap L) = TN \cap TL$). Then $\tilde L = \pi(L)\subseteq \tilde N$ is Legendrian. 
\end{cor}
\begin{proof}
    Let $n+k+1$ be the dimension of $N$, then, $\lorth{TN}$ has dimension $n-k$ by \cref{thm:orth_casos}. Hence, 
    \begin{equation}
        \dim \tilde N  = \dim N - \dim (\lorth{TN}) = 2k + 1.
    \end{equation}

    Since $\tilde L$ is trivially horizontal, we only need to show that 
    \begin{equation}\label{dim_legendriano}
        \dim \tilde L = \dim L\cap N - \dim TL \cap (\lorth{TN}) =  k.
    \end{equation}

    Since by \cref{thm:orth_casos},
    \begin{equation}
        \dim\lorth{(L \cap N)} + \dim (L \cap N)=  2n+1,
    \end{equation}
    and by \cref{thm:orth},
    \begin{equation*}
        \lorth{(L + N)} = \lorth{L}\cap \lorth{N}=  L\cap \lorth{N},
    \end{equation*}
    using the incidence formula,
    \begin{equation}
        \dim (L + \lorth{N}) = \dim L + \dim N - \dim (L \cap N) = 2n+1 +k+ \dim (L \cap N),
    \end{equation}  
     and substituting in \eqref{dim_legendriano} concludes the proof.
\end{proof}

\section{Moment maps}
The moment map is well-known in symplectic geometry. There is a contact analogue~\cite{Albert1989,Loose2001,Willett2002} which has been used to prove reduction theorems via this map. In our proof of this theorem we can see that it can be interpreted as a coisotropic reduction of the level set of the moment map.

We remind that given a Lie group $G$, we denote its Lie algebra by $\lieAlg{g}$ and the dual of its Lie algebra by $\lieAlg{g}^*$.

\begin{defn}\label{def:moment_map}
    Let $(M,\eta)$ be a contact manifold and let $G$ be a Lie group acting on $M$ by contactomorphism. In analogy to the exact symplectic case, we define the moment map $J: M  \to \lieAlg{g}^*$ such that
    \begin{equation}
        J(x)(\xi) = -\eta(\xi_M(x)),
    \end{equation}
    where $x\in M$, $\xi\in\lieAlg{g}$ and $\xi_M\in \VecFields M$ is defined by
    \begin{equation}
        \xi_M(x) = \restr{\pdv{}{t} (\exp(t \xi) \cdot x)}{t=0} 
    \end{equation}
    is the the infinitesimal generator of the action corresponding to $\xi$.
\end{defn}

The moment map has the following properties:
\begin{prop}
    Let $G$ be a Lie group acting by contactomorphisms on a contact manifold $(M,\eta)$. If we let
    \begin{equation}
        \begin{aligned}
            \hat{J}: \mathfrak{g} &\to \Cont^\infty(M)\\
            \xi &\to - \contr{\xi_M} \eta,
        \end{aligned}
    \end{equation}
    so that $\hat{J}(\xi)(x) = J(x)(\xi)$. We obtain that the so-called \emph{moment condition}:
    \begin{equation}
        \dd \hat J  (\xi) = \contr{\xi_M} \dd \eta.
    \end{equation}

    Furthermore
    \begin{equation}
        X_{\hat{J}(\xi)} =  \xi_M.
    \end{equation}
\end{prop}
\begin{proof}
    The fact that $G$ acts by contactomorphisms implies that
    \begin{equation}
        \lieD{\xi_M} \eta = 0.
    \end{equation}
    Thus, by Cartan's formula
    \begin{equation}
        \dd \hat J  (\xi) =  -\dd \contr{\xi_M} \eta = \contr{\xi_M} \dd \eta.
    \end{equation}

    The other equality is a consequence of \cref{thm:hamiltonian_bijection}.
\end{proof}

\begin{prop}\label{thm:moment}
    The moment map defined as above is equivariant under the coadjoint action. That is, for every $g\in G$, the following diagram commutes:
    \begin{equation}
        \begin{tikzcd}
            M  \arrow[r,"g"]  \arrow[d,"J"] &
            M  \arrow[d,"J"] \\ 
            \lieAlg{g}^* \arrow[r,"\Ad_{g^{-1}}^*"] &
            \lieAlg{g}^*
        \end{tikzcd}
    \end{equation}
    where $\Ad^*:G \to \mathrm{GL}(\lieAlg{g}^*)$ is the coadjoint representation, that is, if $g\in G$, $\alpha\in\lieAlg{g}^*$ and $\xi\in\lieAlg{g}$ 
    \begin{equation}
        \Ad_g^*(\alpha)(\xi) = 
        \alpha(\Ad_g(\xi)) = 
        \alpha(T(R_{g^{-1}} L_g)\xi),
    \end{equation}
    and $L_g,R_g:G \to G$ are, respectively, left and right multiplication by $g$.
\end{prop}
\begin{proof}
    We must show
    \begin{equation}
        \hat{J}(\xi)(g x) = \hat{J}(\Ad_{g^{-1}} \xi)(x),
    \end{equation}
    that is,
    \begin{equation}
        (\contr{\xi_M} \eta) (g x) =
         (\contr{{(\Ad_{g^{-1}} \xi)}_M} \eta)(x).
    \end{equation}
    The proof follows from the following identity~\cite[Prop.~4.1.26]{Abraham1978}, which is true for any smooth action
    \begin{equation}
        {(\Ad_{g^{-1}} \xi)}_M = g^* \xi_M,
    \end{equation}
    together with the fact that $g$ preserves the contact form. 
\end{proof}

\begin{lem}\label{thm:moment_coisotropic}
    Let $(M,\eta)$ be a contact manifold on which a Lie group $G$ acts by contactomorphisms. Let $\mu \in \lieAlg{g}^*$ be a regular value of the moment map $J$. Then, for all $x\in J^{-1}(\mu)$
    \begin{equation}
        T_x(G_\mu x) = T_x (G x) \cap T_x (J^{-1}(\mu)),
    \end{equation}
    where $G_\mu = \set{g\in G \mid \Ad^*_{g^{-1}} \mu = \mu}$ is the isotropy group of $\mu$ with respect to the coadjoint action.

    It is also true that
    \begin{equation}
        T_x(J^{-1}(\mu)) = \dhorth{T_x(G x)}.
    \end{equation}

    In particular, if $G=G_\mu$, then $T_x (G x) \subseteq T_x (J^{-1}(\mu))$ and $T_x (J^{-1}(\mu))$ is coisotropic and consists of vertical points. Furthermore
    \begin{equation}
        \lorth{T_x(J^{-1}(\mu))} = T_x(G x)
    \end{equation}
\end{lem}
\begin{proof}
    In~\cite[Cor.~4.1.22]{Abraham1978} we see that
    \begin{equation}
        T_x(G x) = \set{\xi_M(x) \mid \xi \in \lieAlg{g}}.
    \end{equation}

    If $\lieAlg{g}_\mu \subseteq \lieAlg{g}$ denotes the Lie subalgebra corresponding to the Lie subgroup $G_\mu \subseteq G$,  we conclude that $\xi_M(x) \in T_x(G_\mu x)$ if and only if $\xi \in \lieAlg{g}_\mu$. By $\Ad^*$-equivariance, one deduces that
    \begin{equation}
        T_x(J(\xi_M(x))) = \xi_{\lieAlg{g}^*}(\mu),
    \end{equation}
    thus $\xi_M \in T_x(J^{-1}(\mu))= \ker T_x J$ if and only if $\xi_{\lieAlg{g}^*}(\mu)=0$, which means that $\mu$ is a fixed point of $\Ad^*_{\exp(-t \xi)}$ or, equivalently, $\exp(\xi)\in G_\mu$ which, by basic Lie group theory, is the same as $\xi \in \lieAlg{g}_\mu$.

    For the second part, remember (\cref{thm:moment}) that if $\xi \in \lieAlg{g}$ and $v \in T_x M$, then
    \begin{equation}
        \dd \eta(\xi_M(x),v) = \dd \hat{J}(\xi)(x)(v) = T_x J (v)(\xi).
    \end{equation}
    Thus $v\in T_x(J^{-1}(\mu))= \ker T_x J$ if and only if $\dd\eta(\xi_M(x),v) = 0$ for all $\xi \in \lieAlg{g}$ . That is, $T_x(J^{-1}(\mu)) = \dhorth{T_x(G x)} = \dhorth{\set{\xi_M(x) \mid \xi \in \lieAlg{g}}}$.

    In the case $G=G_\mu$, we note that, because $G$ acts by contactomorphisms, $T_x(G x) \subseteq \HorD$, thus, by \cref{thm:orth_rel} we see that
    \begin{equation}
        \lorth{T_x(J^{-1}(\mu))} = T_x(G x) \subseteq T_x (J^{-1}(\mu)).
    \end{equation}
\end{proof}

\begin{thm}[Reduction via moment map]\label{thm:contact_reduction}
    Let $(M,\eta)$ be a contact manifold on which a Lie group $G$ acts freely and properly by contactomorphisms and let $J$ be the moment map. Let $\mu\in \lieAlg{g}$ be a regular value of $J$ which is a fixed point of $G$ under the coadjoint action. Then, $M_\mu = J^{-1}(\mu)/G$ has a unique contact form $\eta_\mu$ such that
    \begin{equation}
        \pi_\mu^* \eta_\mu = \iota_\mu^* \eta,
    \end{equation}
    where $\pi_\mu : J^{-1}(\mu) \to M_\mu$ is the canonical projection and $\iota_\mu: J^{-1}(\mu) \to M$ is the inclusion.
    
    Also the Reeb vector field of the quotient $\Reeb_\mu = \pi_\mu^* \Reeb$ is the projection of the Reeb vector field of $(M,\eta)$.
\end{thm}

\begin{proof}
    This follows from combining \cref{thm:coisotropic_reduction} and \cref{thm:moment_coisotropic}. Since $J^{-1}(\mu)$ is a coisotropic manifold (since $\mu$ is a regular value, its preimage is a manifold) the quotient by its characteristic distribution has a unique contact structure projected from $M$. Also, the quotient of a manifold by a free and proper group action is again a manifold. Both quotients coincide because $\lorth{T_x(J^{-1}(\mu))} = T_x(G x)$, so the leaves of the coisotropic distribution coincide with the orbits of $G$.
\end{proof}

\begin{thm}[Contact Hamiltonian system reduction]
    Let $G$ be a group acting freely and properly by contactomorphisms on $(M, \eta, H)$ such that $H$ is $G$-invariant (that is $H \comp g = H$ for all $g\in G$). Then, with the notations of previous theorem,  $(M_\mu,\eta_\mu, H_\mu)$ is a Hamiltonian system  where $H_\mu$ is projection of $H$  by the action of $G$. This situation is illustrated by the following diagram,
    \begin{equation}
        \begin{tikzcd}
            M \arrow[dr, "H"] & \\
            J^{-1}(\mu) \arrow[u,"\iota_\mu"] \arrow[d,"\pi_\mu"] &
            \RR \\
            M_\mu \arrow[ur," H_\mu"] &,
        \end{tikzcd}
    \end{equation}
    Furthermore ${\pi_\mu}_* \restr{X_H}{J^{-1}(\mu)}= X_{H_\mu}$.
\end{thm}
\begin{proof}
    The fact that $(M_\mu,\eta_\mu, H_\mu)$ is a Hamiltonian system is a consequence of \cref{thm:contact_reduction}. We note that $H_\mu$ is well-defined because $H$ is $G$-invariant.
    
    Now we need to see that $\restr{X_H}{N} \in \VecFields (J^{-1}(\mu))$. Since $H$ is $G$-invariant, for all $g\in G$ we have
    \begin{equation}
        \xi_M (H) = \iota_{\xi_M} \dd H = 0,
    \end{equation}
    that is, $\dd H_x \in \ann{(T_x G x)}$, for all $x\in J^{-1}(\mu)$ or, equivalently, $\lsharp \dd H \in \lsharp \ann{(T_x G x)} = \lorth{(T_x G x)}$. Hence,
    \begin{equation}
        {(X_H)}_x = \lsharp \dd H_x - H(x)\Reeb_x \in  \lorth{(T_x G x)} \oplus \VertD = \dhorth{(T_x G x)}=  T_x(J^{-1}(\mu)),
    \end{equation}
    where the last equality is due to \cref{thm:moment_coisotropic}.

    We remark that ${\pi_\mu}_* \restr{X_H}{J^{-1}(\mu)}$ is well-defined, since both $H$ and $\eta$ are preserved by the action of $G$. We now will show that ${\pi_\mu}_* \restr{X_H}{J^{-1}(\mu)}$ equals $X_{H_\mu}$. We shall denote by $\flat_\mu$ to the $\flat$ isomorfism (\cref{def:sharp_contact}) corresponding to the contact structure in the quotient.
   \begin{align*}
       \flat_\mu (\Xmu) &= 
       \contr{\Xmu} \dd \eta_\mu + (\contr{\Xmu}\eta_\mu) \eta_\mu \\ &=
       - \dd H_\mu + (\Reeb_\mu  H_\mu) \eta_\mu + H_\mu \eta_\mu,
   \end{align*}
   since $\pi_\mu^* \eta_\mu = \restr{\eta}{J^{-1}(\mu)}$ by \cref{thm:contact_reduction}. Hence, $\Xmu$ is the Hamiltonian vector field for $\Xmu$.
\end{proof}

\begin{note}[Lifting solutions]\label{rem:lift_sol}
    A solution to the reduced problem can be lifted to a solution of the initial system~\cite{Marsden1974}. That is, any integral curve $[c(t)]$ for $X_{H_\mu}$is the projection unique integral curve $c(t)$ for $X_H$ after choosing a base point $c(0) = x \in M$. To see that, we pick a curve $d(t)$ such that $d(0)=x$, $[d(t)] = [c(t)]$, that is, $c(t)= g(t) d(t)$ with $g(t) \in G$. We can find that $g(t)$ by solving the following equation
\begin{equation}
    X_H(d(t)) =  d'(t) + (TL_{{g(t)}^{-1}})_P(d(t)),
\end{equation}
which can be seen to have a unique solution by solving
\begin{equation}
    \xi_P(d(t)) = X_H(d(t)) - d'(t),
\end{equation}
for $\xi(t) \in \lieAlg{g}$ and then, we solve
\begin{equation}
    g'(t)= TL_{g(t)} \xi(t)
\end{equation}
for $g(t)$.
\end{note}

\section*{Acknowledgements}
This  work  has  been  partially  supported  by  MINECO  Grants  {MTM2016-76-072-P} and the ICMAT Severo Ochoa projects SEV-2011-0087 and SEV-2015-0554.  Manuel Lainz wishes to thank ICMAT and UAM for a Severo Ochoa master grant and a FPI-UAM predoctoral contract.

\printbibliography    
\end{document}